\documentclass[10pt,twoside,a4paper,reqno]{amsart}
\usepackage{fix-cm}
\usepackage[T1]{fontenc}
\usepackage{amscd,epsfig}
\usepackage{graphics,graphicx}

\newcommand \bydef {\stackrel{\mbox{\scriptsize def}}{=}}
\theoremstyle{plain}
\newtheorem{theorem}{Theorem}[section]
\newtheorem{proposition}[theorem]{Proposition}
\newtheorem{corollary}[theorem]{Corollary}
\newtheorem{lemma}[theorem]{Lemma}

\newtheorem*{Theorem}{Theorem}
\newtheorem*{Proposition}{Proposition}

\theoremstyle{remark}

\newtheorem{remark}[theorem]{Remark}

\theoremstyle{definition}

\numberwithin{equation}{section}

\newcommand{\splitatcommas}[1]{%
  \begingroup
  \ifnum\mathcode`,="8000
  \else
  \begingroup\lccode`~=`, \lowercase{\endgroup
    \edef~{\mathchar\the\mathcode`, \penalty0 \noexpand\hspace{0pt plus 1em}}%
  }\mathcode`,="8000
  \fi
  #1%
  \endgroup
}

\newcommand{\al}{\alpha}

\newcommand{\ga}{\gamma}
\newcommand{\Ga}{\Gamma}

\newcommand{\bC}{\mathbb C}
\newcommand{\bCP}{\mathbb {CP}}
\newcommand{\bR}{\mathbb R}

\newcommand{\bZ}{\mathbb Z}
\newcommand{\ZZ}{\mathcal Z}

\newcommand{\A}{\mathcal A}

\newcommand{\cG}{\mathcal G}

\newcommand{\eps}{\epsilon}

\newcommand{\RR}{\mathcal R}

\newcommand{\si}{\sigma}

\newcommand \F {\mathfrak F}
\newcommand \prt {\partial}

\newcommand \DD {\mathfrak D}

\newcommand \AD {\mathcal{AD}}

\newcommand \zz {\mathbf z}
\newcommand \bzz {\bar{\mathbf z}}
\newcommand \mR {\mathfrak R}
\newcommand \mP {\mathfrak P}
 
\newcommand \wMM{\mathcal M}
\newcommand \mO{\mathcal O}
\newcommand{\BM}{\phi}
\newcommand {\fM}{\mathfrak M} 
\newcommand{\TT}{\mathcal T}

\newcommand{\DT}[1]{#1 \dots #1}
\def \lmod#1\rmod {\vphantom{#1}\left|\smash{#1}\right|}
\newcommand{\name}[1]{\operatorname{\mathrm #1}\nolimits}
\def \id {\name{id}}
\newcommand{\widevec}[1] {\overset{\longrightarrow}{#1}}

\begin{document}

\title[Algebraic relations between moments of plane polygons]
{Algebraic relations between harmonic and anti-harmonic moments of plane polygons}

\author[Yu.\,Burman]{Yurii Burman}
\address{National Research University Higher School of Economics
  (NRU-HSE), Russia, and Independent University of Moscow,
  B.~Vlassievskii per., 11, Moscow, 119002, Russia}
\email{burman@mccme.ru}

\author [R.\,Fr\"oberg]{Ralf  Fr\"oberg}
\address{Department of Mathematics, Stockholm University, SE-106 91
  Stockholm, Sweden}
\email{ralff@math.su.se}

\author[B.\,Shapiro]{Boris Shapiro}
\address{Department of Mathematics, Stockholm University, SE-106 91
  Stockholm, Sweden}
\email{shapiro@math.su.se}

\dedicatory{To Isaac Schoenberg and Theodore Motzkin whose insights
  laid the foundations of this topic}

\date{\today}

\keywords{harmonic and anti-harmonic moments, polygonal measures}
\subjclass[2010]{Primary 44A60; Secondary 31B20}

\begin{abstract}  
  In this paper we describe the algebraic relations satisfied by the
  harmonic and anti-harmonic moments of simply connected, but not
  necessarily convex planar polygons with a given number of vertices.
\end{abstract}
  
\maketitle 

\section{Introduction and main results}\label{s1}

\subsection{Basic notions and background}

Let $\mu$ be a a finite compactly supported Borel measure in the plane
$\bR^2 = \bC$. For  $j = 0, 1, \dots$, its $j$-th {\em harmonic moment} is a
complex number given by:
%*
\begin{equation*}
  m_j(\mu) \bydef  \int_{\bC} z^j\,d\mu(z).
\end{equation*}
%*
Analogously, its $j$-th {\em anti-harmonic moment} is given by:
%*
\begin{equation*}
  \bar m_j(\mu) \bydef  \int_{\bC} \bar z^j\,d\mu(z).
\end{equation*}
%*
A function 
%*
\begin{equation*}
  \mathfrak u_\mu(z) \bydef \int_\bC \ln \lmod z-\xi\rmod\, d\mu(\xi) 
\end{equation*}
%*
is called the {\em logarithmic potential} of $\mu$. It is harmonic
outside the support of $\mu$ and well-defined almost everywhere in
$\bC$. The germ of $ \mathfrak u_\mu(z) $ at $\infty$ is determined by
the sequence of harmonic moments $\{m_j(\mu)\}_{j=0}^\infty$: the
Taylor expansion at $\infty$ of the {\em Cauchy transform} of $\mu$
defined as
%*
\begin{equation*}
  \mathfrak C_\mu(z)\bydef \int_\bC \frac{d\mu(\xi)}{z-\xi} =
  \frac{\mathfrak \prt \mathfrak u_\mu(z)}{\prt z}, 
\end{equation*}
%*
is
%*
\begin{equation*}
  \mathfrak C_\mu(z) = \frac{m_0(\mu)}{z} + \frac{m_1(\mu)}{z^2} +
  \frac{m_2(\mu)}{z^3} + \dots .
\end{equation*}
%*

The problem of the recovery of a measure from its logarithmic
potential at $\infty$ (alias {\em ``the inverse problem in logarithmic
  potential theory''}) is a classical area of potential analysis going
back to the early 1920s and still quite active.  One of its milestones
is the fundamental paper \cite{Nov}, in which P.S.\,Novikov proved
that Lebesgue measures of two different star-shaped %(e.g., convex)
domains cannot have the same logarithmic potential near $\infty$. In
other words, sequences of harmonic moments of Lebesgue measures of two
star-shaped plane domains cannot coincide.

For non-star-shaped domains a similar statement is false: see e.g.
\cite[p.~333]{broSt} for examples of pairs of non-convex polygons
having the same logarithmic potentials near $\infty$.

In this century the problem reappeared in mathematical physics in
connection with integrable systems and the Hele-Shaw flow, see
e.g.\ \cite {MWZ,2DToda}. In particular, in \cite{2DToda} S.\,Natanzon
and A.\,Zabrodin extended Novikov's result showing that harmonic
moments can be used as ``coordinates'' on the set of all star-shaped
domains. Their results imply that for any sequence of numbers
$\alpha_0, \alpha_1 \dots$, there exists a star-shaped domain whose
Lebesgue measure $\mu$ satisfies the conditions $m_j(\mu) = \alpha_j$,
$j = 0, 1, \dots$.

This claim is no longer true if one considers moments of Lebesgue
measures of polygons with a fixed number $n$ of vertices. The harmonic
moments (and anti-harmonic) moments of such polygons are algebraically
dependent and the main goal of the present article is to describe
these dependencies.

\subsection{The main object: polygonal measures, an explicit formula and complexification}

A plane polygon $P \subset \bC$ is determined by its sequence of
vertices $z_1 \DT, z_n \in \bC$ ordered counterclockwise, but not
every sequence of $n$ points in $\bC$ is a sequence of vertices for
some simply connected polygon. It is natural to generalize the notion
of the Lebesgue measure of a polygon as follows.

Let ${\bf a}=(a_1 \DT, a_n),\; a_j \in \bC^2$ be a sequence of points
with $a_j = (x_j, y_j)$, $j = 1 \DT, n$.  Instead of $x_j$ and $y_j$
we are going to use more convenient coordinates $z_j = x_j+ iy_j$ and
$\bar z_j = x_j - iy_j$. For brevity, denote $\zz \bydef (z_1 \DT,
z_n)$ and $\bzz \bydef (\bar z_1 \DT, \bar z_n)$.

If all $a_j \in \bR^2 \subset \bC^2$ (that is, all $x_j$ and $y_j$ are
real) then every $\bar z_j$ is indeed the complex conjugate of $z_j$;
if we identify $\bR^2$ with $\bC$ as $(x,y) \leftrightarrow x+iy$ then
$a_j$ becomes $z_j$. We will call such situation ``the case of real
vertices''. In general, though, $\zz$ and $\bzz$ are $n$-tuples of
independent complex variables.

Define an {\em oriented closed polygonal curve} $ \Ga_{\bf a}$ by
%*
\begin{equation*}
  \Ga_{\bf a} \bydef \widevec{[a_1,a_2]} \DT\cup \widevec{[a_{n-1},a_n]}
  \cup \widevec{[a_n,a_1]}.
\end{equation*}
%*
Fix an auxiliary convex polygon $P_* \subset \bR^2$ with the vertices
$w_1 \DT, w_n$ ordered counterclockwise, and let $\TT$ be its {\em
  triangulation}, i.e.\ a set of diagonals of $P_*$ having no common
internal points and cutting $P_*$ into triangles. Let $F_{\bf a,\TT}:
P_* \to \bC^2$ be the map sending every $w_j$ to $a_j$ and affine on
every triangle of the triangulation. (It is easy to show that $F_{\bf
  a,\TT}$ exists and is unique and continuous.)

The image $\Delta_{\bf a,\TT} = F_{\bf a,\TT}(P_*) \subset \bC^2$ is a
polygonal disk in $\bC^2$ bounded by $\Ga_a$. The disk $\Delta_{\bf
  a,\TT} \subset \bC^2$ is piecewise immersed (though $F_{\bf a,\TT}$
is not always an immersion); for generic $\bf a$ this disk is
embedded.

The disk $\Delta_{\bf a,\TT}$ supports the measure $\mu_{\bf a,\TT} =
(F_{\bf a,\TT})_* dxdy$ which is the direct image of the Lebesgue
measure on $P$ under $F_{\bf a,\TT}$; we call $\mu_{\bf a,\TT}$ a {\em
  polygonal measure}. The disk $\Delta_{\bf a,\TT}$ itself depends on
the triangulation $\TT$, but certain integrals with respect to
$\mu_{\bf a,\TT}$ do not.

\def \labelenumi {(\theenumi)}
\def \theenumi {\roman{enumi}}
\makeatletter
\def \p@enumi\csname #1\endcsname{(\theenumi)}
\makeatother
\begin{theorem}\label{Th:TriaIndep}
  \begin{enumerate}
  \item\label{It:TriaIndep} Let $h: \bC^2 \to \bC$ be a holomorphic
    function of two variables. Then the integral $\int_{\Delta_{\bf
        a,\TT}} h\, d\mu_{\bf a,\TT}$ does not depend on the
    triangulation $\TT$ and is equal to $\int_{\Ga_{\bf a}} \omega$,
    where $\omega$ is any $1$-form such that $d\omega = -\frac{i}{2}
    h\,dz \wedge d\bar z$.
  \item\label{It:Real} If the vertices $a_1 \DT, a_n$ are real then
    $\mu_{\bf a,\TT}$ is independent of $\TT$. It is supported on a
    compact subset of $\bR^2$ and its density at a point $q \in \bR^2
    \setminus \Ga_{\bf a}$ equals the linking number of the $1$-cycle
    $\Ga_{\bf a} \subset \bR^2$ with the $0$-cycle $q - \infty$. In
    particular, if $a_1 \DT, a_n \in \bR^2$ are vertices of a simply
    connected polygon listed counterclockwise then $\mu_{a,\TT}$ is
    the Lebesgue measure of this polygon.
  \end{enumerate}
\end{theorem}

The theorem follows easily from the Stokes' theorem; see Section
\ref{Sec:HarmOnly} for a detailed proof. The last claim in assertion
\ref{It:Real} explains the term ``polygonal measure''.

\begin{corollary}[of assertion \ref{It:TriaIndep}]\label{Cr:MomentIndep}
  The harmonic moment $m_j(\mu_{{\bf a},\TT})$ of the polygonal
  measure does not depend on the triangulation $\TT$.
\end{corollary}

The following lemma is straightforward:
\begin{lemma}\label{Lm:Add}
  Assume that the triangulation $\TT$ contains the diagonal $(1,m)$
  which divides $\TT$ into two parts, $\TT'$ and $\TT''$. If $a = (a_1
  \DT, a_n)$, $a' = (a_1 \DT, a_m)$ and $a'' = (a_1, a_m, a_{m+1} \DT,
  a_n)$, then $\Delta_{\bf a,\TT} = \Delta_{\bf a',\TT'} \cup
  \Delta_{\bf a'',\TT''}$ and $\mu_{\bf a,\TT} = \mu_{\bf a',\TT'} +
  \mu_{\bf a'',\TT''}$.
\end{lemma}
\begin{corollary}\label{Cr:SumTria}
  If $\TT$ consists of the edges $(1,3), (1,4) \DT, (1,n-1)$, then
  %*
  \begin{equation*}
    \mu_{\bf a,\TT} = \mu_{a_1a_2a_3} + \mu_{a_1a_3a_4} \DT+ \mu_{a_1a_{n-1}a_n}.
  \end{equation*}
  %* 
\end{corollary}
(The right-hand side of the latter formula contains measures supported
on triangles, so there is no need to specify a triangulation). Notice
that if the vertices are real then, up to a sign, $\mu_{a_1 a_k
  a_{k+1}}$ is the Lebesgue measure of a (possibly degenerate)
triangle with the vertices $a_1$, $a_k$ and $a_{k+1}$. If this
triangle is degenerate then $\mu_{a_1 a_k a_{k+1}}$ is zero, otherwise
the sign is taken to be $+$ if the triangle is oriented
counterclockwise, and $-$ if clockwise.

Take again ${\bf a} = (a_1 \DT, a_n), a_j = (x_j,y_j) \in \bC^2$; $z_j
= x_j + iy_j, \bar z_j = x_j - iy_j$. Obviously, $z_j$ and $\bar z_j$
determine $a_j$ since $x_j = (z_j + \bar z_j)/2$ and $y_j = (z_j -
\bar z_j)/2i$. Set
%*
\begin{equation*}
  \nu_k(\zz,\bzz) \bydef \binom{k}{2} m_{k-2}(\mu_{\bf
    a,\TT})\quad\text{and}\quad \bar \nu_k(\zz,\bzz) = \binom{k}{2}
  \bar m_{k-2}(\mu_{\bf a,\TT}).
\end{equation*}
%*
In particular, $\nu_0 = \nu_1 = \bar \nu_0=\bar \nu_1=0$ for all
$(\zz, \bzz)$. By Corollary \ref{Cr:MomentIndep} both sides of the
equalities are independent of $\TT$.

\begin{remark}
  The index shift $(k-2) \mapsto k$ used above is convenient since the
  normalized moments $\nu_k$ and $\bar \nu_k$ are homogeneous of
  degree $k$ with respect to the dilatations on $\bC^2$. In other
  words, $\nu_k(t\zz,t\bzz) = t^k \nu_k(\zz,\bzz)$ where $t\zz \bydef
  (tz_1 \DT, tz_n)$.
\end{remark}

The following theorem provides explicit formulas for $\nu_k(\zz,\bzz)$
and $\bar \nu_k(\zz,\bzz)$.

\begin{theorem}\label{Th:Formula}
  For any positive integer $k\ge 2$ one has
  %*
  \begin{multline}\label{eq:expl}
    \nu_k(\zz,\bzz) = \frac{i}{4} \sum_{j=1}^n(\bar z_j-\bar
    z_{j+1})\frac{z_j^k-z_{j+1}^k}{z_j-z_{j+1}}\\
    = \frac{i}{4} \sum_{j=1}^n(\bar z_j-\bar
    z_{j+1})(z_j^{k-1}+z_j^{k-2}z_{j+1}+\dots+z_{j+1}^{k-1}).
  \end{multline}
  %*
  The anti-harmonic moment $\bar \nu_k(\zz,\bzz)$ is given by the same
  formula \eqref{eq:expl} with $\zz$ and $\bzz$ interchanged. 
\end{theorem}

Observe that, up to the factor $i/4$, each normalized harmonic and
anti-harmonic moment of a polygonal measure is a polynomial with
integer coefficients in the variables $z_1 \DT, z_n$ and $\bar z_1
\DT, \bar z_n$. Notice additionally that $\bar \nu_2(\zz,\bzz) =
-\nu_2(\zz,\bzz)$.

\subsection{Main results about the relations between the moments}

In this paper, the problem of describing the algebraic relations among
the moments of polygons will be understood in two different ways which
we call the {\em algebraic} and the {\em geometric} approaches
respectively. The algebraic approach amounts to finding the algebraic
relations between the polynomials $\nu_j(\zz,\bzz)$, $j = 2, 3,
\dots$, while the geometric approach deals with finding algebraic
relations including both $\nu_j(\zz,\bzz)$ and $\bar \nu_j(\zz,\bzz)$,
$j = 2, 3, \dots$.

In case of the algebraic approach our main result is relatively
simple. Namely, all harmonic moments can be expressed as rational
functions of the first $2n-2$ moments $\nu_2(\zz,\bzz) \DT,
\nu_{2n-1}(\zz,\bzz)$, and these $2n-2$ moments are algebraically
independent. More precisely, denote by $\F_n$ the field extension of
$\bC$ generated by the sequence of polynomials
$\{\nu_j(\zz,\bzz)\}_{j=2}^\infty$.

\begin{theorem}\label{Th:HarmOnly}
  \begin{enumerate}
    \item \label{It:FieldOfRation} $\F_n =
      \bC(\nu_2,\dots,\nu_{2n-1})$ and is isomorphic to the field of
      rational functions in $2n-2$ independent complex variables.
    \item\label{It:ContainsSymm} $\F_n\supset \bC(\zz)^{S_n}$, where
      $\bC(\zz)^{S_n}$ is the field of symmetric rational functions of
      $z_1 \DT, z_n$.
  \end{enumerate}
\end{theorem}

Explicit formulas expressing harmonic moments $\nu_j(\zz,\bzz)$ with
$j\ge 2n$ via the first $2n-2$ are given by rational
functions. However all their denominators are powers of one fixed
polynomial $\DD_n$, the determinant of the matrix \eqref{eq:Toeplitz}
below. In fact, if one considers the {\em ring extension} $\RR_n$ of
$\bC$ (as opposed to field) generated by the sequence of polynomials
$\{\nu_j(\zz,\bzz)\}_{j=2}^\infty$, the situation is as follows.

\begin{theorem}\label{Th:HarmOnlyRing}
  \begin{enumerate}
    \item\label{It:InfiniteGen} The ring $\RR_n = \bC[\nu_2,\nu_3,
      \dots]$ is not generated by any finite collection of harmonic
      moments $\nu_2 \DT, \nu_N$.
    \item\label{It:Localize} For the polynomial $\DD_n \in
      \bC[\zz,\bzz]$ given by the determinant of \eqref{eq:Toeplitz},
      the localization $\left.\RR_n\right|_{\textstyle \DD_n}$ is isomorphic to
      $\bC[\nu_2,\dots, \nu_{2n-1}] \left[\dfrac{1}{\DD_n}\right]$.
  \end{enumerate}
\end{theorem}

Notice that $\RR_n$ does not contain the ring $\bC[\zz]^{S_n}$ of symmetric
polynomials in the variables $z_1 \DT, z_n$ as a subring since the 
expression of the basic (e.g.\ elementary) symmetric polynomials via
$\nu_2 \DT, \nu_{2n-1}$ involves division by some powers of $\DD_n$.

Further, in the geometric approach we consider the field extension
$\widetilde\F_n$ of $\bC$ generated by both sequences
$\{\nu_j\}_{j=2}^\infty$ and $\{\bar\nu_j\}_{j=2}^\infty$. (Recall
that $\nu_2=-\bar \nu_2$). Here the situation is more complicated.

\begin{theorem}\label{Th:Both}
  \begin{enumerate}
    \item\label{It:Gen} The field $\widetilde\F_n$ is generated by the
      first $4n-5$ harmonic and anti-harmonic moments $\nu_2, \nu_3,
      \bar \nu_3 \DT, \nu_{2n-1}, \bar \nu_{2n-1}$.
    \item\label{It:SymmSubF} The field $\widetilde\F_n$ contains a
      subfield $H = \bC(\zz,\bzz)^{S_n \times S_n}$ of rational
      functions symmetric with respect to two groups of variables $z_1
      \DT, z_n$ and $\bar z_1 \DT, \bar z_n$ separately.
    \item\label{It:GenNu2} $\widetilde\F_n$ is an algebraic extension
      of $H$ generated by the single element $\nu_2$. The degree $d_n$
      of this extension equals $n!(n-1)!$ if $n$ is odd and
      $2((n-1)!)^2$ if $n$ is even.
  \end{enumerate}
\end{theorem}

\begin{remark}
  Notice that any algebraic extension of a field of characteristics
  zero is generated by a single element. So the essence of assertion
  \ref{It:GenNu2} of Theorem~\ref{Th:Both} is that a specific element
  $\nu_2$ is a generator.  We describe the Galois closure of this
  extension and its Galois group later in Section \ref{sec:galois}.
\end{remark}

Algebraic relations between the usual moments of polygons and
polytopes in several special situations were discussed in a recent
(joint with K.~Kohn and B.~Sturmfels) paper \cite{KoShSt} of the third
author. For algebraic domains, the relations between the moments were
studied in e.g.\ \cite[Section 3]{LaPu}.

The structure of the paper is as follows. Section \ref{Sec:HarmOnly}
contains detailed proofs of Theorems \ref{Th:TriaIndep},
\ref{Th:Formula}, \ref{Th:HarmOnly}, and \ref{Th:HarmOnlyRing}, as
well as some formulas related to the logarithmic potential. In Section
\ref{Sec:Both} we describe the action of the group $S_n \times S_n$ on
the field $\widetilde\F_n$ and prove Theorem \ref{Th:Both}. In Section
\ref{sec:galois} we describe the Galois group of the (Galois closure
of the) extension $\bC \subset \widetilde\F_n$. Section
\ref{sec:triangles} contains explicit description of this extension
and its Galois group in the simplest case $n=3$ (i.e.\ for
triangles). We finish the paper with some questions expressing our
outlook to the further development of the subject, see
Section~\ref{sec:outlook}.

\subsection*{Acknowledgements.} The authors want to thank
D.~Pasechnik, J.-B.~Lasserre, and, especially B.~Sturmfels for their
interest in this topic and discussions. We are also grateful to
S.~Lvovskiy for his help with some aspects of algebraic geometry,
including the algebro-geometric part of the proof of
Theorem~\ref{Th:Both}. The research of the first author was funded by
the Russian Academic Excellence Project `5-100' and by the Simons--IUM
fellowship 2019 by the Simons Foundation. The third author expresses
his gratitude to MPI MIS in Leipzig and to Simons Institute for the
theory of computing in Berkeley for the hospitality in June 2018 and
April 2019. He also wants to acknowledge the financial support of his
research provided by the Swedish Research Council grant 2016-04416.

\section{Proofs and additional results}\label{Sec:HarmOnly}

\begin{proof}[Proof of Theorem \ref{Th:TriaIndep}]
  To settle assertion \ref{It:TriaIndep}, notice first that since
  $h(z,\bar z)$ is holomorphic in both variables, the $2$-form
  $h(z,\bar z)\,dz \wedge d\bar z$ is closed and therefore exact:
  $h(z,\bar z)\,dz \wedge d\bar z = d\omega$ for some $1$-form
  $\omega$. If $n=3$ then there exists only one triangulation; the
  equality $-\frac{i}{2}\int_{\Delta_a} h\,dz \wedge d\bar z =
  \int_{\Ga_a} \omega$ follows from the Stokes' theorem.

  Next apply the formulas for $n=3$ to each triangle of the
  triangulation $\TT$ and add them. Observe that the integrals over
  the internal diagonals of the triangulation contribute two equal
  terms with the opposite signs and therefore cancel, while the
  integrals over the ``sides of the polygon'' (i.e., segments of
  $\Ga_{\bf a}$) appear once and survive in the total sum. On the
  other hand, all the maps $F_{\bf a,\TT}$ are the same on the sides
  of the polygon $P$ and map them to the same line $\Ga_{\bf a}$;
  thus, the integrals over $\Ga_{\bf a}$ are the same for all possible
  triangulations, which proves assertion \ref{It:TriaIndep} for any
  positive integer $n\ge 3$.

  Assertion \ref{It:Real} is evident for $n=3$. Indeed, for $q$ lying
  outside the triangle $\Ga_{\bf a}$, the linking number of $q -
  \infty$ with $\Ga_{\bf a}$ is $0$, and for $q$ lying inside the
  triangle, it equals $\pm 1$; the choice of the sign depends on the
  orientation of the triangle. Notice now that the map sending
  $\Ga_{\bf a}$ to the polygonal measure is additive, i.e.  if ${\bf
    a} = (a_1 \DT, a_n)$, ${\bf a}' = (a_1 \DT, a_m)$, ${\bf a}'' =
  (a_1, a_m, a_{m+1} \DT, a_n)$ then $\Ga_{\bf a} = \Ga_{\bf a'} +
  \Ga_{\bf a''}$ as $1$-cycles, and therefore the linking number of $q
  - \infty$ with $\Ga_{\bf a}$ is equal to the sum of its linking
  numbers with $\Ga_{\bf a'}$ and $\Ga_{\bf a''}$. This observation
  together with Lemma \ref{Lm:Add} allow us to finish the proof by
  induction on $n$.
\end{proof}

Now we are ready to derive an explicit formula for the moments of
polygonal measures.

\begin{proof}[Proof of Theorem \ref{Th:Formula}]
  For a harmonic moment $\nu_k$ we are in the situation of Theorem
  \ref{Th:TriaIndep} with $h(z,\bar z) = \frac{ik}{4}(k-1)z^{k-2}$, so
  one can take $\omega = kz^{k-1}\,dz$.

  Parametrize a segment $[p,q] \subset \bC$ as $z = p(1-t) + tq$, $t
  \in [0,1]$, then
  %*
  \begin{align*}
    \frac{ik}{4} \int_{[p,q]} z^{k-1}\,d\bar z &= \frac{ik}{4}
    \int_0^1 (p(1-t)+tq)^{k-1} (\bar q - \bar p)\,dt\\    
    &= \frac{i}{4} \int_0^1 \frac{\bar q - \bar p}{q-p}\,
    d((p(1-t)+tq)^k) = \frac{i}{4} (\bar q - \bar
    p)\frac{q^k-p^k}{q-p},
  \end{align*}
  %*
  Equation \eqref{eq:expl} follows now from Theorem
  \ref{Th:TriaIndep}.
\end{proof}

Define now, following \cite{PS}, the {\em normalized generating
  function $\Psi_\mu(w)$ for harmonic moments} of a measure $\mu$ as
%*
\begin{equation}\label{eq:Psi}
  \Psi_\mu(w) \bydef \sum_{j=2}^\infty \nu_j(\mu) w^{j-2} = \sum_{j=0}^\infty \binom {j+2}{2} m_j(\mu)w^j.
\end{equation}
%*
$\Psi_\mu(w)$ is closely related to the Cauchy transform $\mathfrak
C_\mu(z)$ at $\infty$. Namely,
%*
\begin{equation*}
  \Psi_\mu(w) = \frac{1}{2}\frac{d^2}{dw^2}\left(\sum_{j=0}^\infty m_j(\mu) w^{j+2}\right).
\end{equation*}
%*
At the same time for a compactly supported measure $\mu$ and
sufficiently large $\lmod z\rmod$ one has $w\mathfrak C_\mu(z) =
\sum_{j=0}^\infty m_j(\mu)/ z^j$. Hence, if $\lmod w\rmod$ is
sufficiently small, then
%*
\begin{equation*}
  \Psi_\mu(w) = \frac{1}{2}\frac{d^2}{dw^2}\left(w\mathfrak C_\mu\left(\frac{1}{w}\right)\right).
\end{equation*}
%*
Similar multivariate generating functions were recently considered in
\cite{GPSS}.

For $n=3$ we will sometimes denote
%*
\begin{equation*}
  D_{i,j,k} \bydef \nu_2(z_i,z_j,z_k,\bar z_i,\bar z_j,\bar z_k).
\end{equation*}
%*
where the subscripts $i, j, k$ vary as appropriate. Note that, for any
$n$, if the vertices are real (that is, $\zz$ and $\bzz$ are complex
conjugate) and the points $z_1 \DT, z_n$ are vertices of a simply
connected polygon then the moment $\nu_2(\zz,\bzz)$ is equal to the
(signed) area of the polygon; so $D_{i,j,k}$ is the area of a triangle
with the vetices $z_i, z_j, z_k$.

Important in our consideration is the following observation which can
be found in \cite{PS}:

\begin{proposition}\label{Pp:Psi}
 %*
 \begin{equation}\label{Eq:Psi}
   \Psi_{\mu_\zz}(w) = \frac{D_{123}}{(1-z_1w)(1-z_2w)(1-z_3w)}.
 \end{equation}
 %*
\end{proposition}

\begin{proof}
  Substitution of \eqref{eq:expl} into \eqref{eq:Psi} gives
  %*
  \begin{align*}
    \Psi_{\mu_\zz}(w) &= \frac{i}{4}\frac{\bar z_1-\bar z_2}{z_1-z_2}
    \sum_{k=2}^\infty (z_1^k-z_2^k)w^{k-2} + \text{cyclic}\\
    &= \frac{i}{4} \biggl(\frac{\bar z_1-\bar z_2}{z_1-z_2}
    \bigl(\frac{z_1^2}{1-z_1w} - \frac{z_2^2}{1-z_2w}\bigr) +
    \text{cyclic}\biggr)\\    
    &= \frac{i}{2} \frac{z_1 \bar z_2 - z_2 \bar z_1 + z_2 \bar z_2 -
      z_2 \bar z_3 + z_3 \bar z_1 - z_1 \bar
      z_3}{(1-z_1w)(1-z_2w)(1-z_3w)}\\
    &= \frac{D_{123}}{(1-z_1w)(1-z_2w)(1-z_3w)}.
  \end{align*}
  %*
  (where ``+cyclic'' means the sum of two extra summands
  cyclically shifting $z_1 \mapsto z_2 \mapsto z_3 \mapsto z_1$).
\end{proof}

Proposition \ref{Pp:Psi} and Corollary \ref{Cr:SumTria} imply the following. 

\begin{corollary}\label{cor:triv}
  For any $\zz = (z_1 \DT, z_n)$, one has
  %*
  \begin{equation*}
    \Psi_{\mu_Z}(w) = \sum_{j=2}^{n-1} \frac{D_{1,j,j+1}} {(1-z_1w)(1-z_j w)(1-z_{j+1}w)}.
  \end{equation*}
  %*
\end{corollary}

\begin{corollary}[of Corollary \ref{cor:triv}]
  For every $\zz = (z_1 \DT, z_n)$, there exists a unique polynomial
  $\AD_\zz(w)$ of degree at most $n-3$ such that
  %*
  \begin{equation}\label{eq:PsiN}
    \Psi_{\mu_Z}(w) = \frac{\AD_\zz(w)}{\prod_{j=1}^n(1-z_j w)},
  \end{equation}
  %*
\end{corollary}

\begin{remark}
  Properties of the polynomial $\AD_\zz(w)$ were studied in detail in
  \cite{Wa} (see also \cite{War}). In particular, the following was
  proved:
  \begin{Proposition}[\cite{Wa}]
  \begin{enumerate}
  \item\label{It:Leading} The coefficient at $w^{n-3}$ in $\AD_\zz$
    equals $\nu_2(\zz,\bzz)$.

  \item For generic $\zz$ there exists a polynomial $A_\zz(u,v)$
    (called the adjoint polynomial of $\zz$), unique up to a
    multiplicative constant, vanishing at all points $\ell \in
    (\bCP^2)^*$ corresponding to lines joining $z_k$ and $z_\ell$ with
    $\lmod k-\ell\rmod \ge 2$. Then one has $\AD_\zz(w) = A_\zz(w,iw)
    \times const.$ where constant is determined by assertion
    \ref{It:Leading}.
  \end{enumerate}
  \end{Proposition}
\end{remark}

Next we settle Theorem~\ref{Th:HarmOnly}. 
   
\begin{proof}[Proof of Theorem~\ref{Th:HarmOnly}]
  To prove assertion \ref{It:ContainsSymm} notice that the field
  $\bC(\zz)^{S_n}$ of invariant rational functions is generated by the
  elementary symmetric polynomials $e_1(\zz) \DT, e_n(\zz)$. Thus it
  is enough to show that they belong to $\F_n$.

  It is a well-known fact, see, e.g., \cite[Th.~4.1.1]{St}, that if
  %*
  \begin{equation*}
    \sum_{j=0}^\infty f(j)t^j = \frac{P(t)}{Q(t)},
  \end{equation*}
  %*
  where $\deg P <\deg Q$ and $Q(t) = 1+\al_1 t+\al_2 t^2 \DT+ \al_d
  t^d$ with $\al_d\neq 0$, then for all $k\ge 0$, one has the
  recurrence relation:
  %*
  \begin{equation*}
    f(k+d)+\al_1f(k+d-1)+\al_2f(k+d-2) \DT+ \al_d f(k) = 0.
  \end{equation*}
  It follows from \eqref{eq:PsiN} that
  %*
  \begin{multline*}
    \AD_Z(w) = \prod_{j=1}^n(1-z_jw) \cdot \sum_{j=0}^\infty
    \nu_{j+2}(\zz)w^j\\    
    = (1-e_1(\zz)w + \dots + (-1)^n e_n(\zz)w^n) \cdot
    \sum_{j=0}^\infty \nu_{j+2}(\zz)w^j.
  \end{multline*}
  %*
  Since $\deg \AD_z \ge n-3$, then for every $k \ge -2$, one has
  %*
  \begin{equation}\label{eq:rec}
    \nu_{k+2+n}(\ZZ) - e_1(\ZZ)\nu_{k+1+n}(\ZZ) +
    e_2(\zz)\nu_{k+n}(\zz) - \dots + (-1)^ne_n(\zz)\nu_{k+2}(\zz) = 0.
  \end{equation}
  %*

  Consider the first $n$ equations of the recurrence \eqref{eq:rec}. This linear
  system has the form:
  %*
  \begin{equation}\label{eq:sys}
    U \cdot E =V,
  \end{equation}
  %*
  where $U$ is the  Toeplitz $n\times n$-matrix given by
  %*
  \begin{equation}\label{eq:Toeplitz}
    U = \begin{pmatrix}
      \nu_{n-1}(\zz,\bzz) & \nu_{n-2}(\zz,\bzz) & \dots & \nu_1(\zz,\bzz) &
      \nu_0(\zz,\bzz)\\
      \nu_{n}(\zz,\bzz) & \nu_{n-1}(\zz,\bzz) & \dots & \nu_2(\zz,\bzz) &
      \nu_1(\zz,\bzz)\\
      \vdots & \vdots & \ddots & \vdots & \vdots\\
      \nu_{2n-2}(\zz,\bzz) & \nu_{2n-3}(\zz,\bzz) & \dots & \nu_{n}(\zz,\bzz) &
      \nu_{n-1}(\zz,\bzz)
    \end{pmatrix},
  \end{equation}
  %*
  and $E$ and $V$ are column vectors of length $n$ given by
  %*
  \begin{equation*}
    E = (e_1(\zz), -e_2(\zz) \DT, (-1)^{n+1}e_n(\zz))^T,
  \end{equation*}
  %*
  and
  %*
  \begin{equation*}
    V = (\nu_{n}(\zz,\bzz), \nu_{n+1}(\zz,\bzz) \DT,
    \nu_{2n-1}(\zz,\bzz))^T.
  \end{equation*}
  %*
  (Recall that $\nu_0 = \nu_1 = 0$). Assuming that $U$ is invertible,
  one obtains $E = U^{-1}V$ which means that every $e_j(\zz)$ is
  expressed as a rational function of the normalized moments
  $\nu_2(\zz,\bzz) \DT,, \nu_{2n-1}(\zz,\bzz)$ with the fixed
  denominator equal to the determinant of $U$. Thus assertion
  \ref{It:ContainsSymm} of Theorem~\ref{Th:HarmOnly} is proved.

To settle assertion \ref{It:FieldOfRation} we argue as follows.  Using the recurrence
relation \eqref{eq:rec} one can express every $\nu_{k+2+n}(\zz,\bzz)$,
$k\ge n-2$, as a rational function of the first $2n-2$ normalized
harmonic moments $\nu_2 \DT, \nu_{2n-1}$, which proves that $\F_n$ is
generated by the elements $\nu_2(\zz,\bzz) \DT, \nu_{2n-1}(\zz,\bzz)$.

Now for $s = 1, 2, \dots$, denote by $h_s(\zz)$ the $s$-th complete
symmetric polynomial of $z_1 \DT, z_n$. The identity
%*
\begin{equation}
  1 = \prod_{\ell=1}^n (1-w z_\ell) \times \frac{1}{\prod_{\ell=1}^n
    (1-w z_\ell)} = \sum_{m-0}^n (-1)^m e_m(\zz)w^m \sum_{k=0}^\infty
  h_k(\zz) w^k,
\end{equation}
%*
implies the standard relation $\sum_{m=0}^n (-1)^m e_m(\zz)
h_{j-m}(\zz) = 0$ for all $j = 1, 2, \dots$. 
Corollary \ref{cor:triv} implies that
%*
\begin{align}
  \sum_{j=0}^\infty \nu_{j+2}(\zz,\bzz) w^{j} &=
  \biggl(\sum_{s=0}^\infty h_s(\zz) w^s\biggr) \sum_{\ell=1}^{n-1}
  D_{1,\ell,\ell+1}(\zz,\bzz)\nonumber\\
  &\hphantom{\biggl(\sum}\times (1-wz_2) \dots
  \widehat{(1-wz_\ell)} \widehat{(1-wz_{\ell+1})} \dots
  (1-wz_n) \label{Eq:MViaQ}\\
  &= \biggl(\sum_{s=0}^\infty h_s(\zz) w^s\biggr)
  \sum_{\ell=1}^{n-1} D_{1,\ell,\ell+1}(\zz,\bzz)\nonumber\\
  &\hphantom{\biggl(\sum}\times \sum_{m=0}^{n-3} (-1)^m w^m e_m(z_2
  \DT, \widehat{z_\ell}, \widehat{z_{\ell+1}} \DT, z_n) \nonumber\\
  &= \biggl(\sum_{s=0}^\infty h_s(\zz) w^s\biggr) \sum_{m=0}^{n-3}
  w^m Q_m(\zz,\bzz), \nonumber
\end{align}
%*
where 
%*
\begin{equation}\label{Eq:QViaD}
  Q_m(\zz,\bzz) \bydef \sum_{\ell=1}^{n-1} (-1)^\ell
  D_{1,\ell,\ell+1}(\zz,\bzz) e_m(z_2 \DT, \widehat{z_\ell},
  \widehat{z_{\ell+1}} \DT, z_n).
\end{equation}
%*
This equation gives
%*
\begin{equation}\label{Eq:MViaH}
  \nu_j(\zz,\bzz) = \sum_{m=0}^{n-3} (-1)^m
  Q_m(\zz,\bzz)h_{j-m-2}(\zz),
\end{equation}
%*
where $j = 2, 3, \dots$, and $h_k(\zz) \bydef 0$ for $k < 0$. 

Multiplying \eqref{Eq:MViaQ} by $\prod_{k=1}^n (1-wz_k) =
\sum_{\ell=0}^n (-1)^\ell e_\ell(\zz)w^\ell$, one obtains
%*
\begin{equation}\label{Eq:AllMViaEQ}
  \sum_{\ell=0}^n (-1)^\ell e_\ell(\zz)w^\ell \sum_{j=2}^\infty
  \nu_j(\zz,\bzz) w^{j-1} = \sum_{m=0}^{n-3} (-1)^m w^m
  Q_m(\zz,\bzz).
\end{equation} 
%*
Relation \eqref{Eq:MViaQ} leads to 
%*
\begin{equation*}
  \mathfrak F_n = \bC(e_1 \DT, e_n, Q_0 \DT, Q_{n-3}).
\end{equation*}
%*
We are going to show that $Q_0 \DT, Q_{n-3}$ are algebraically
independent over \linebreak $\bC(e_1 \DT, e_n)$. Indeed, the functions
$z_1 \DT, z_n$ are the roots of
%*
\begin{equation*}
  t^n - e_1(\zz) t^{n-1} + \dots +(-1)^n e_n(\zz) = 0
\end{equation*}
%*
and, therefore, belong to an algebraic extension $AE_n$ of
$\bC(e_1 \DT, e_n)$. The same field $AE_n$ contains the
functions $e_m(z_2 \DT, \widehat{z_k}, \widehat{z_{k+1}} \DT, z_n)$
mentioned in \eqref{Eq:QViaD}, which implies that $D_{123} \DT,
D_{1,n-1,n} \in AE_n$ as well. Thus, the field $\bC(e_1 \DT,
e_n,$ $ D_{123} \DT, D_{1,(n-1),n})$ is an algebraic extension
of $\mathfrak F_n$.

  The polynomial $D_{123}$ depends on the variables $z_2$ and $\bar
  z_2$, while $D_{1,k,k+1}$ for all $k = 3 \DT, (n-1)$ do not. Take
  any $u \in \bC$ and substitute $\bar z_2 \mapsto \bar z_2+u$,
  leaving all  other $\bar z_j$ and all $z_j$ unchanged. This
  operation preserves the values of $e_1(\zz) \DT, e_n(\zz)$ as well as the values 
of  $D_{134} \DT, D_{1,(n-1),n}$. On the other hand, $D_{123}$
  takes infinitely many values as $u$ varies, and therefore it is not
  a root of any algebraic equation with the coefficients dependent
  only on $e_1(\zz) \DT, e_n(\zz)$ and $D_{134}\DT,
  D_{1,(n-1),n}$. Now take $z_3, \bar z_3$ and consider $D_{134}$ to
  conclude that it is algebraically independent of $D_{145} \DT,
  D_{1,(n-1),n}$, etc.

  In this way we prove that $D_{123} \DT, D_{1,(n-1),n}$ are
  algebraically independent elements over the field $\bC(e_1 \DT,
  e_n)$. Hence in the tower of extensions
  %*
  \begin{multline}
    \bC(e_1 \DT, e_n) \subset \mathfrak F_n = \bC(e_1 \DT, e_n,
    Q_0 \DT, Q_{n-3})\\
    \subset \bC(e_1 \DT, e_n, D_{123} \DT, D_{1,(n-1),n})
  \end{multline}
  %*
  the transcendence degree of the last field over the first one equals  $n-2$. The second extension is algebraic, and therefore,
  the first extension has the transcendence degree $n-2$ as
  well. Consequently, $Q_0 \DT, Q_{n-3}$ are algebraically independent
  over $\bC(e_1 \DT, e_n)$. Equation \eqref{Eq:MViaH} implies
  that $\nu_2 \DT, \nu_n$ are algebraically independent as well. Since
  $e_1(\zz) \DT, e_n(\zz)$ are also algebraically independent, the
  field $\mathfrak F_n = \bC(e_1 \DT, e_n,\nu_2 \DT, \nu_n)$ is
  isomorphic to the field of rational fractions in $2n-2$ independent variables.

  On the other hand, $\bC(e_1 \DT, e_n) \subset \bC(\nu_2 \DT,
  \nu_{2n-1}) \subset \mathfrak F_n$, which gives
  %*
  \begin{equation*}
    \mathfrak F_n = \bC(\nu_2 \DT, \nu_{2n-1}).
  \end{equation*}
  %*
  Theorem~\ref{Th:HarmOnly} is proved.
\end{proof}

\begin{remark}
  Equation \eqref{eq:sys} coincides with \cite[equation (2.8)]{GMV} if
  one reads the coefficient vector $E$ backwards and reflects $U$ with
  respect to its vertical midline.
\end{remark}

On the other hand, the ring extension $\RR_n = \bC[\nu_2, \nu_3,
  \dots]$ behaves somewhat differently:

\begin{proof}[Proof of Theorem \ref{Th:HarmOnlyRing}]
  By \eqref{eq:expl} each polynomial $\nu_j(\zz,\bzz)$ is a linear
  function in the variables $\bzz = (\bar z_1 \DT, \bar z_n)$, and
  thus cannot be a polynomial function of the other $\nu_k$ of degree
  exceeding one. But on the other hand, $\nu_j$ is homogeneous of
  degree $j-1$ with respect to $\zz = (z_1 \DT, z_n)$, so such linear
  dependence is impossible as well. This proves assertion
  \ref{It:InfiniteGen} of the theorem.

  Further, observe that for $j>2n-1$, the recurrence relation
  \eqref{eq:rec} allows us to express each $\nu_j$ as a rational
  function of $\nu_2,\dots, \nu_{2n-1}$ whose denominator is a power
  of $\DD_n \bydef \det U$, where $U$ is defined by the equation
  \eqref{eq:Toeplitz}. This fact is equivalent to assertion
  \ref{It:Localize}.
\end{proof}

\section{Anti-harmonic moments and Galois group}\label{Sec:Both}
The group $S_n \times S_n$, where $S_n$ is the usual symmetric group
on $n$ elements, acts on the field of rational functions
$\bC(\zz,\bzz)$ permuting the variables: the first copy of $S_n$ acts
on $z_1 \DT, z_n$ while the second copy acts on $\bar z_1 \DT, \bar
z_n$. We denote the action of a pair $(\si,\tau)$, $\si, \tau \in S_n$
on a rational function $R$, by the subscript:
%*
\begin{equation*}
  R(\zz,\bzz)_{(\sigma\tau)} \bydef R(z_{\sigma(1)} \DT,
  z_{\sigma(n)}, \bar z_{\tau{(1)}} \DT, \bar z_{\tau(n)}).
\end{equation*}
%*
In particular, ${\nu_j}_{(\sigma,\tau)}$ is the result of the
permutation of variables in the $j$-th normalized moment
$\nu_j(\zz,\bzz)$. Note that ${\nu_j}_{(\sigma,\tau)}(\zz,\bzz)$ is
also the $j$-th normalized moment of the ${\bf a} = (a_1 \DT, a_n)$,
where $a_j = (x_j,y_j)$ with $z_{\si(j)} = x_j + iy_j, \bar
z_{\tau(j)} = x_j - iy_j$ (in other words, $a_j = ((z_{\si(j)}+\bar
z_{\tau(j)})/2, (z_{\si(j)}-\bar z_{\tau(j)})/(2i))$). If $\si = \tau$
then $a_1 \DT, a_n \in \bC^2$ are actually the same points as the
vertices of the polygonal line for $\nu_j(\zz,\bzz)$, but ordered
differently. In particular, if all the vertices for $\nu_j(\zz,\bzz)$
are real, then the vertices for ${\nu_j}_{(\si,\si)}(\zz,\bzz)$ are
real, too.

Let us compute now the stabilizer of the above $S_n \times
S_n$-action. To do this, we will especially need to describe the
$S_n\times S_n$-orbit of the lowest degree moment $\nu_2(\zz,\bzz)$.

Observe that every ${\nu_2}_{(\si,\tau)}(\zz,\bzz)$ is a bilinear form
in the variables $(\zz,\bzz)$; denote by $M_{(\si,\tau)}$ its matrix in
a standard basis of $\bC^n$.

Theorem~\ref{Th:Formula} provides  that
%*
\begin{equation}\label{eq:matrix}
  M_{(\id,\id)}=\begin{pmatrix}
     0&1&0&0&\dots &0&-1\\
    -1& 0& 1&0&\dots&0&0\\
     0&-1&0&1&\dots&0&0\\
     \vdots&\vdots&\vdots&\vdots&\vdots&\vdots&\vdots\\
     1&0&0&\dots &0&-1&0
           \end{pmatrix}.
\end{equation}
%*
Further, for an arbitrary pair $(\si,\tau)$, the matrix
$M_{(\si,\tau)}$ is obtained from \eqref{eq:matrix} by permuting the
rows of $M_{(\id,\id)}$ according to the permutation $\si$ and
permuting the columns, according to $\tau$. (These two permutation
actions commute.)

Expression for $M_{(\si,\tau)}$ can be written in different
terms. Namely, let $\mP: \bC[S_n] \to \name{Mat}(n,\bC)$ be the
standard permutation representation of the symmetric group: i.e. for
any $\sigma \in S_n$, the $(i,j)$-th entry of the $n \times n$-matrix
$\mP[\sigma]$ is $1$ if $j = \sigma(i)$ and $0$ otherwise. Then
$M_{(\id,\id)} = \mP[C-C^{-1}]$, where $C = (1,2, \dots, n)$ is the
long cycle sending $j$ to $j+1$ for $j=1,\dots, n-1$ and sending $n$
to $1$. Therefore $M_{(\si,\tau)} = \mP[\si]^* \mP[C-C^{-1}]
\mP[\tau]$. The group representation $\mP$ is orthogonal, that is,
$\mP[\si]^* = \mP[\si^{-1}]$ for any $\si \in S_n$, so eventually
%*
\begin{equation}\label{eq:ActionMatr}
  M_{(\si,\tau)} = \mP[\si^{-1} (C-C^{-1}) \tau].
\end{equation}
%*

Denote by $\wMM(n)$ the set of all $n \times n$-matrices such that
every its row and every column contains one entry equal to $1$,
another entry equal to $-1$, and all the remaining entries vanish. The
group $S_n \times S_n$ acts on $\wMM(n)$ by permutation of the rows
and the columns. The set of all matrices $M_{(\si,\tau)} \in \wMM(n)$
for $(\si,\tau) \in S_n \times S_n$ is an orbit of this action which
we denote by $\mO_n$.

The stabilizer of $\nu_2$ under the $S_n \times S_n$-action is the
stabilizer of $\mO_n$ which we denote by $\cG$. We want to describe
$\cG$ explicilty. Recall that $C \bydef (1,2,\dots,n)$ denotes the
long cycle. If $n=2\ell$ is even then set $C_1 \bydef (1,3,\dots,
2\ell-1)$, $C_2 \bydef (2,4,\dots, 2\ell)$, $\delta_1 \bydef
(1,2)(3,4) \dots (2\ell-1,2\ell)$, and $\delta_2 \bydef (2,3)(4,5)
\dots (2\ell,1)$.

The next lemma is straightforward:

\begin{lemma}\label{Lm:EvenRel}
  For $n$ even, the following relations hold:
  \begin{enumerate}
    \item $C^2 = C_1\cdot C_2 = C_2 \cdot C_1$,
    \item $C \cdot C_1 = C_2 \cdot C$ and $C \cdot C_2 = C_1\cdot C$,
    \item $\delta_1 \cdot C_1 = C_2 \cdot \delta_1$, $\delta_1\cdot
      C_2 = C_1\cdot \delta_1$, $\delta_2 \cdot C_1 = C_2 \cdot
      \delta_2$ and $\delta_2 \cdot C_2 = C_1 \cdot \delta_2$,
    \item $\delta_2 \cdot C = C \cdot \delta_1$ and $\delta_1 \cdot C
      = C \cdot \delta_2$.
  \end{enumerate}
\end{lemma}

\begin{proposition}\label{Pp:11}
  \begin{enumerate}
  \item For $n$ odd, the stabilizer $\cG \bydef \name{St}(\mO_n) \subset
    S_n\times S_n$ coincides with the cyclic group $\mathbb Z_n$
    generated by $(C,C) \in S_n \times S_n$.

  \item For $n=2\ell$ even, the stabilizer $\cG \subset S_n\times S_n$
    consists of all elements $(C_1^u \cdot C_2^v, C_1^v\cdot C_2^u) \in
    S_n \times S_n$ and of all elements $(\delta_1 \cdot C_1^u \cdot
    C_2^v, \delta_2 \cdot C_1^v \cdot C_2^u) \in S_n \times S_n$,
    where $u,v=0,\dots, \ell-1$. As an abstract group, $\cG$ is
    non-commutative, but it contains an index $2$ subgroup isomorphic
    to ${\mathbb Z_\ell} \times {\mathbb Z_\ell}$.
  \end{enumerate}
\end{proposition}

\begin{proof}
  An immediate check using \eqref{eq:ActionMatr} shows that 
  $M_{(\si,\tau)}= M_ {(C \cdot \si, C \cdot \tau)}$ for all $n$. (It
  also  follows from the obvious fact that the polygonal line $\Ga_a$
  does not change under the cyclic shift of the points $a_1 \mapsto a_2
  \DT\mapsto a_n \mapsto a_1$.) For even $n = 2\ell$, the same formula
  and Lemma \ref{Lm:EvenRel} additionally imply that $M_{(\si,\tau)}=
  M_ {(C_1\cdot \si, C_2\cdot \tau)}$, $M_{(\si,\tau)}= M_{(C_2\cdot
    \si, C_1\cdot \tau)}$ and $M_{(\si,\tau)} = M_{(\delta_1\cdot \si,
    \delta_2 \cdot\tau)}$.

  Assume now that
  \begin{equation}\label{eq:extra}
    M_{(\si,\tau)}=M_{(\si^\prime,\tau^\prime)}.
  \end{equation}
  Observe that relabelling the variables one can, without loss of
  generality, choose $\si = \tau = \id$, where $\id$ is the identity
  permutation. In the representation-theoretical notation, formula
  \eqref{eq:extra} is equivalent to $\mP[\sigma^{-1} (C - C^{-1})
    \tau] = \mP[C - C^{-1}]$, that is, $\mP[(C - C^{-1}) \tau] =
  \mP[\sigma (C - C^{-1})]$.

  If $u_1 \DT, u_n \in \bC^n$ is the standard basis, then the latter
  equation means that $$\mP[(C - C^{-1}) \tau](u_i) = u_{\tau(i)+1} -
  u_{\tau(i)-1} = \mP[\sigma (C - C^{-1})](u_i) = u_{\si(i+1)} -
  u_{\si(i-1)}$$ for all $i = 1 \DT, n$. In other words, $\tau(i)+1 =
  \sigma(i+1)$ for all $i$, that is, $\tau\cdot C=C\cdot \si$ and
  $C\cdot \tau=\si\cdot C$. These relations imply
  %*
  \begin{equation}\label{eq:add} 
    C^2\cdot\tau = C\cdot\si\cdot C = \tau\cdot C^2.
  \end{equation}
  %*
For $n=2\ell+1$, one has
%*
\begin{equation*}
  (C^2)^{\ell+1}\cdot\tau=\tau\cdot (C^2)^{\ell+1} \Leftrightarrow C\cdot \tau=\tau\cdot C.
\end{equation*}
%*
Since $C$ only commutes with its own powers, one obtains $\tau=C^k$
for some $k=0,\dots, n-1$ implying that $\si=C^k=\tau$.

Consider now the case $n=2\ell$. Then $C^2 = C_1 \cdot C_2$ (a product
of two independent cycles). Set $\mathcal E_1 = \{1, 3 \DT,
2\ell-1\}$ and $\mathcal E_2 = \{2, 4 \DT, 2\ell\}$.  % of $\{1, 2 \DT,n\}$. 
Since $\tau$ commutes with $C_1\cdot C_2$, and the subgroups of
$S_n$ generated by $C_1$ and $C_2$ act transitively on $\mathcal E_1$
and $\mathcal E_2$ respectively, one has that either $\tau(\mathcal E_1) =
\mathcal E_1$ and $\tau(\mathcal E_2) = \mathcal E_2$ or
$\tau(\mathcal E_1) = \mathcal E_2$ and $\tau(\mathcal E_2) = \mathcal
E_1$.

In the first case the restrictions of $\tau$ to $\mathcal E_1$ and
$\mathcal E_2$ commute with the cycles $C_1$ and $C_2$, respectively,
and therefore $\tau = C_1^u \cdot C_2^v$ and $\sigma = C\cdot \tau
\cdot C^{-1} = C_1^v \cdot C_2^u$. In the second case the same
reasoning holds for the permutations $\tilde\tau \bydef \delta_2\cdot
\tau$, so $\tau = \delta_2 \cdot C_1^u \cdot C_2^v$ and $\sigma = C
\cdot \tau \cdot C^{-1} = \delta_1 \cdot C_1^v\cdot C_2^u$.
\end{proof} 

Proposition \ref{Pp:11} and assertion \ref{It:GenNu2} of Theorem
\ref{Th:Both} imply the following claim:

\begin{corollary}
  $\widetilde\F_n \subset \bC[\zz,\bzz]^\cG$, where $\cG \subset S_n
  \times S_n$ is the stabilizer group of $\nu_2$ described in
  Proposition \ref{Pp:11}.
\end{corollary}

Let us now settle Theorem \ref{Th:Both}.

\begin{proof}
  Assertions \ref{It:Gen} and \ref{It:SymmSubF} are proved similarly
  to the corresponding statements in Theorem \ref{Th:HarmOnly} about
  the field $\F_n$.

  To prove assertion \ref{It:GenNu2} set
  %*
  \begin{equation}\label{Eq:MinimalPol}
    P(t) \bydef  \prod_{(\sigma, \tau)\in (S_n\times S_n)/\cG}(t-{\nu_2}_{(\si, \tau)}).
  \end{equation}
  %*
  Here the index $(\si, \tau)$ runs over a system of representatives
  of the right cosets of $S_n\times S_n$ with respect to the
  stabilizer subgroup $\cG$. Thus $\deg P$ is equal to the number of
  these right cosets, that is, to $(n!)^2/\vert \cG\vert$. By
  Proposition~\ref{Pp:11}, one has $\deg P = n!(n-1)!$ for $n$ odd and
  $\deg P = 2((n-1)!)^2$ for $n$ even.

  Let us prove that $P(t)$ is the minimal polynomial defining $\nu_2$
  over the field $\bC(\zz, \bzz)^{S_n\times S_n}$; it is enough to
  show that $P$ is irreducible.
  
  Indeed, assume that $P(t)$ is reducible and $Q(t)=\prod_{(\si,
    \tau)\in U}(t-{\nu_2}_{(\si, \tau)})$ is its irreducible factor
  where $U$ is some proper subset of $(S_n\times S_n)/\cG$. Thus $\deg
  Q = \# U < \# (S_n\times S_n)/\cG$. The coefficients of the
  polynomial $Q$ are $S_n \times S_n$-invariant, so for any $(\si,
  \tau)\in U$, the element ${\nu_2}_{(\si, \tau)}$ must be a root of
  $Q$. By Proposition~\ref{Pp:11}, this implies that $U$ intersects
  any right coset in $(S_n\times S_n)/\cG$. Thus $\deg Q \ge \#
  (S_n\times S_n)/\cG$. Contradiction.
  
  Now let us show that for generic $(\zz,\bzz)$, all the roots of the
  polynomial $P$ defined by  \eqref{Eq:MinimalPol} are  simple.

  \begin{lemma}\label{Lm:AllSimple}
    For generic $(\zz, \bzz)$, the values of all bilinear forms
    ${\nu_2}_{(\si,\tau)}(\zz,\bzz)$ are pairwise distinct, where
    $(\si,\tau)$ runs over all right cosets $(S_n \times S_n)/\cG$
    with respect to the stabilizer group $\cG$.
  \end{lemma}

  \begin{proof}%[Proof of the lemma]
    Indeed, if it is not the case, then $\bC^n \times \bC^n$ is a
    union of finitely many sets $L_{\si,\tau,\si',\tau'} \bydef
    \{(\zz,\bzz) \mid {\nu_2}_{(\si,\tau)}(\zz,\bzz) =
    {\nu_2}_{(\si',\tau')}(\zz,\bzz)\}$. The functions
    ${\nu_2}_{(\si,\tau)}$ are bilinear forms, so
    $L_{\si,\tau,\si',\tau'}$ are quadrics. A vector space over
    $\bC$ cannot be a union of finitely many nontrivial quadrics, so
    $L_{\si,\tau,\si',\tau'} = \bC^n \times \bC^n$ for some
    $\si,\tau,\si',\tau'$. But then $(\si,\tau) = (\si',\tau') \bmod
    \cG$, which contradicts to the choice of $(\si,\tau)$ and
    $(\si',\tau')$ (one element from every right coset). The lemma
    follows.
  \end{proof}

  Fix some generic $c_1 \DT, c_n$ and  $d_1 \DT, d_n$. Then the set
  %*
  \begin{equation*}
    \{(\zz,\bzz) \mid e_j(\zz) = c_j, e_j(\bzz) = d_j, j = 1 \DT, n\}
  \end{equation*}
  %*
  is a generic $S_n \times S_n$-orbit in $\bC^n \times \bC^n$, where
  $e_j$ is the $j$-th elementary symmetric function. By
  Lemma~\ref{Lm:AllSimple}, the values of $\nu_2$ at different points
  of the orbit are distinct, so the values $e_1(\zz) \DT, e_n(\zz),
  e_1(\bzz) \DT, e_n(\bzz)$ and $\nu_2(\zz,\bzz)$ determine the point
  $(\zz, \bzz)$ completely, and therefore, determine the values
  $\nu_j(\zz,\bzz)$ for all $j = 3, 4, \dots$.

  Fix some $j$, and let $Y \subset \bC^{2n+1}$ be the closure of the
  set
  %*
  \begin{equation*}
    \{(e_1(\zz) \DT, e_n(\zz), e_1(\bzz) \DT, e_n(\bzz),
    \nu_2(\zz,\bzz)) \in \bC^{2n+1} \mid \zz, \bzz \in \bC^n\}.    
  \end{equation*}
  %*
  Introduce the algebraic variety
  %*
  \begin{multline*}
    \Ga \bydef \{(y,c) \in Y \times \bC \mid \exists \zz, \bzz \in
    \bC^n: y = (e_1(\zz) \DT, e_n(\bzz), \nu_2(\zz,\bzz)), c =
    \nu_j(\zz,\bzz)\}\\
    \subset \bC^{2n+1}.
  \end{multline*}
  %*
  By Lemma \ref{Lm:AllSimple}, the projection map $p: \Gamma \to
  \bC^{2n+1}$ given by $p(y,c) \bydef y$ is generically one-to-one
  onto its image. Hence there exists a rational map $\tilde R:
  \bC^{2n+1} \to \Ga$ such that $\left.\tilde R\right|_{p(\Ga)} =
  p^{-1}$, see \cite{Mumford}. If $R \bydef q \circ \tilde R$, where
  $q: \bC^{2n+1} \times \bC \to \bC$ is the standard projection then
  $\nu_j(\zz,\bzz) = R(e_1(\zz) \DT, e_n(\zz), e_1(\bzz) \DT,
  e_n(\bzz), \nu_2(\zz,\bzz))$.

  Thus we have shown that
  %*
  \begin{equation*}
    \nu_j \in \bC(e_1(\zz) \DT, e_n(\zz), e_1(\bzz) \DT,
    e_n(\bzz))(\nu_2) = \bC(\zz,\bzz)^{S_n \times S_n}(\nu_2).
  \end{equation*}
  %*
  On the other hand, it follows from assertion \ref{It:Gen} that
  %*
  \begin{equation*}
    \widetilde\F_n = \bC(e_1(\zz) \DT, e_n(\bzz), \nu_2, \nu_3, \bar
    \nu_3 \DT, \nu_{2n-1}, \bar \nu_{2n-1}),
  \end{equation*}
  %*
  which implies $\widetilde\F_n = \bC(e_1(\zz) \DT, e_n(\bzz), \nu_2)
  = \bC(\zz,\bzz)^{S_n \times S_n}(\nu_2)$.
\end{proof} %of Theorem \ref{Th:Both}

For any nonnegative integer $k$, denote by $Y_k \subset \bC^{2n+2k+1}$ the closure of the set
%*
\begin{multline*}
  \{(e_1(\zz) \DT, e_n(\zz), e_1(\bzz) \DT, e_n(\bzz),
  \nu_2(\zz,\bzz), \nu_3(\zz,\bzz), \bar \nu_3(\zz,\bzz) \DT,\\
  \nu_{k-2}(\zz,\bzz), \bar \nu_{k-2}(\zz,\bzz)) \mid \zz, \bzz \in
  \bC^n\} \subset \bC^{2n+2k+1}.
\end{multline*}
%*
Using this notation, assertion \ref{It:GenNu2} of Theorem
\ref{Th:Both} can be reformulated as follows.

\begin{corollary}
  For every positive integer $k$, the variety $Y_k$ is birationally
  equivalent to $Y_0$.
\end{corollary}

\begin{remark}
  Denote by $\widetilde\RR_n \bydef \bC[\nu_2, \nu_3, \bar \nu_3,
    \dots]$ the ring extension generated by all harmonic and
  anti-harmonic moments. Although $\bC(\zz, \bzz)^{S_n\times S_n}
  \subset \widetilde\F_n$, it is not true that $\bC[\zz,
    \bzz]^{S_n\times S_n}\subset \widetilde \RR_n$ because the
  elementary symmetric functions are only expressed as rational
  functions of the moments. On the other hand, the inclusion
  $\widetilde \RR_n\subset\bC[\zz, \bzz]^\cG$ obviously holds.

  Similarly to Theorem \ref{Th:Both} the same circumstance (i.e.\ the
  presence of a denominator in formulas) does not allow us to conclude
  that $\nu_2, \nu_3, \bar\nu_3 \DT, \nu_{2n-1}, \bar\nu_{2n-1}$
  generate $\widetilde\RR_n$. Probably (though we have not yet proved
  this) the situation is similar to assertion \ref{It:InfiniteGen} of
  Theorem \ref{Th:HarmOnlyRing}: the ring cannot be generated by any
  proper subset of $\nu_j, \bar\nu_j$, $j = 2, 3, \dots$. Also we can
  conjecture that an analog of assertion \ref{It:Localize} of the same
  theorem holds: all the denominators in question are powers of a
  single polynomial $\widetilde\DD_n$.
\end{remark}

\begin{remark}
  Formulas \eqref{eq:expl} also show that ${\nu_j}_{(\xi,\xi)} = -\nu_j$
  for all $j$, where $\xi$ is an involution reading the sequence $(12
  \dots n)$ in the opposite direction: $\xi = (1,n)(2,n-1)
  \dots$. Together with a cyclic group $\bZ_n$ with a generator
  $(C,C)$ the involution $\xi$ generates the dihedral group.
\end{remark}

\section{Galois group of the equation satisfied by $\nu_2$}\label{sec:galois}

Assertion \ref{It:GenNu2} of Theorem~\ref{Th:Both} claims that the
minimal polynomial $P(t)$ for the element $\nu_2$ generates the
algebraic extension of the field $\bC[\zz,\bzz]^{S_n\times S_n}$ of
degree
%*
\begin{equation*}
  d_n \bydef \begin{cases}
    n!(n-1)! &\text{if $n$ is odd,}\\
    2((n-1)!)^2 &\text{if $n$ is even.}
  \end{cases}
\end{equation*}
%*
This extension is not Galois; its Galois closure is the field generated
by all the roots of $P$, that is, by ${\nu_2}_{(\si,\tau)}$ for all
$\si, \tau \in S_n$. In this section we calculate the Galois group of
the closure, or, equivalently, the Galois group of the polynomial $P(t)$. To do this,
we need to describe the algebraic dependencies between the polynomials
${\nu_2}_{(\si,\tau)}$; by definition, the Galois group of $P(t)$ is
the subgroup of $S_n \times S_n$ preserving all these dependencies.

Denote by $\fM_n$ the linear span of the set of all $n \times
n$-matrices $M_{(\si,\tau)}$ where $(\si,\tau)\in S_n\times S_n$,  see equation \eqref{eq:matrix} above and
the text following it.

\begin{lemma}\label{lm:span}
  For any $n\ge 3$, the space $\fM_n\subset \name{Mat}_n$ consists of all
  $n \times n$-matrices with vanishing row and column sums for each row
  and column.
\end{lemma} 

\begin{proof}
  Recall that $\mP$ is the standard $n$-dimensional permutation
  representation of the group $S_n$. By equation \eqref{eq:ActionMatr}
  the matrix $M_{(\si,\tau)}$ belongs to the image of $\mP$. The
  representation $\mP$ is reducible; it splits into the trivial
  $1$-dimensional representation in the subspace $V_0 \subset \bC^n$
  spanned by the vector $v_0 = (1, 1 \DT, 1)$ and the irreducible
  representation of dimension $(n-1)$ in the space $V \bydef \{(z_1
  \DT, z_n) \mid z_1 \DT+ z_n = 0\}$. So both $V_0$ and $V$ are
  invariant subspaces of all the $M_{(\si,\tau)}$.
  
  The representation of $S_n$ on $V_0$ is trivial, so
  $\left.C\right|_{V_0} = \id$, so that $M_{(\si,\tau)} = 0$ on the
  space $V_0$. Therefore the sum of matrix elements of
  $M_{(\si,\tau)}$ in every row and column vanishes.

  According to \eqref{eq:ActionMatr} the lemma is equivalent to the
  following statement: for every linear operator $X: V \to V$ there
  exist constants $a_{\si,\tau} \in \bC$ such that $\sum_{\si,\tau \in
    S_n} a_{\si,\tau} \mP[\sigma^{-1} (C - C^{-1}) \tau] = X$. A
  standard  result in representation theory  says that the image of
  an irreducible representation of the group algebra of any finite
  group is the full matrix algebra of the representation space; so,
  for any linear operator $Y: V \to V$ there exist constants $a_\tau,
  \tau \in S_n$ such that $\sum_{\tau \in S_n} a_\tau \mP[\tau] = Y$
  on $V$, see e.g.  \cite{Repr}. Now Lemma~\ref{lm:span} is equivalent to the 
  statement that 
 for any linear operator $X: V \to V$, there exist operators
    $Y_\si, \si \in S_n$ such that 
    %*
    \begin{equation}\label{Eq:SumYSigma}
      \sum_{\si \in S_n} \mP[\si^{-1} (C - C^{-1})] Y_\si = X.
    \end{equation}
    %*

To prove this claim observe that  the operator $M_0 \bydef \mP[C - C^{-1}]: V \to V$ is nonzero; so,
  the linear hull $W \subset V$ of the spaces $\mP[\si^{-1} (C -
    C^{-1})](V) \subset V$, $\si \in S_n$, is $S_n$-invariant and
  nonzero. The representation $V$ is irreducible which implies that $W =
  V$. In other words, there exist not necessarily distinct permutations $\si_1 \DT, \si_{n-1}
  \in S_n$ and vectors $w_1 \DT, w_{n-1}
  \in V$ such that the vectors $v_i = \mP[\si_i^{-1} (C -
    C^{-1})](w_i)$, $i = 1 \DT, n-1$, form a basis in $V$.

  Let  $u_{ij}, i,j = 1 \DT, n-1$ be the constants such that $X(v_j) =
  \sum_{i=1}^{n-1} u_{ij} w_i$ for all $j$. Define  the operators
  $Y_\si, \si \in S_n$ given by:
  %*
  \begin{equation*}
    Y_\si(e_j) = \sum_{i: \si_i = \sigma} u_{ij} w_i. 
  \end{equation*}
  %*
  (if no $\si_i$ equals $\si$, then $Y_\si = 0$ and it does not enter
  \eqref{Eq:SumYSigma}). An immediate check shows that
  \eqref{Eq:SumYSigma} holds and the claim follows. Lemma~\ref{lm:span} is settled. 
\end{proof}

To move further, for $1 \le i,j \le n-1$, denote by $\BM_{ij} \in
\fM_n$ the matrix whose entries equal to $1$ at positions $(i,j)$ and
$(n,n)$, to $-1$ at positions $(i,n)$ and $(n,j)$, and vanish
elsewhere. The next lemma is obvious:

\begin{lemma}\label{cor:dim}
  \begin{enumerate}
    \item For any $n\ge 3$, $\dim \fM_n=(n-1)^2$.
    \item For any $n\ge 3$, matrices $\BM_{i,j}$ with $1\le i\le
      n-1;\, 1 \le j \le n-1$, form a basis in $\fM_n$.
  \end{enumerate}
\end{lemma}

Denote by $\A_n\subset \bC[\zz,\bzz]$ the subalgebra generated by the
bilinear forms ${\nu_2}_{(\si,\tau)}$ with $(\si,\tau)\in
S_n\times S_n$.

\begin{proposition}\label{prop:subalgebra}
  For any $n \ge 3$, one has
  %*
  \begin{equation}
    \A_n \simeq \frac{\bC[D]}{\langle I_2\rangle},
  \end{equation}
  %*
  where $D = (d_{ij})$ is a $(n-1)\times(n-1)$ matrix with variable
  entries $d_{ij}$ and $\langle I_2\rangle $ is the ideal generated
  by all $2 \times 2$-minors of $D$.
\end{proposition}

\begin{proof}
  Observe that $\A_n$ is generated by the linear space of bilinear
  forms whose matrices (written in the basis $(\zz,\bzz)$) belong to
  $\fM_n$. For simplicity, we will identify bilinear forms with their
  matrices and denote this space by $\fM_n$ as well. By
  Lemma~\ref{cor:dim} the $(n-1)^2$ forms $\BM_{ij}$, $1 \le i,j \le
  n-1$ constitute a basis for $\fM_n$. Explicitly
  %*
  \begin{equation*}
    \BM_{i,j}=z_i\bar z_j - z_i\bar z_n - z_n\bar z_j + z_n\bar z_n =
    (z_i - z_n)(\bar z_j - \bar z_n).
  \end{equation*}
  %*
  One can easily check the equalities
  %*
  \begin{equation}\label{Eq:Rel}
    \BM_{i_1,j_1}\BM_{i_2,j_2} - \BM_{i_1,j_2}\BM_{i_2,j_1} = 0
  \end{equation}
  %*
  coming from $2\times 2$-minors.

  So, $\A_n$ is isomorphic to the sub-algebra of $\bC[\zz,\bzz]$
  generated by $\BM_{ij}, 1 \le i,j \le n-1$. The substitution $u_i =
  z_i-z_n, i = 1\dots n-1, u_n = z_n$ and $v_i = \bar z_i-\bar z_n,
  i=1,\dots,n-1, v_n = \bar z_n$ shows that $\A_n$ is isomorphic to
  $\bC[u_iv_j,1\le i,j\le n-1]$.
 
  Now observe that $\bC[u_iv_j,1\le i,j\le n-1]$ is the coordinate
  ring of the Segre embedding $\mathbb {P}^{n-2}\times \mathbb
  {P}^{n-2}\rightarrow \mathbb{P}^{(n-1)^2-1}$, where
  %*
  \begin{equation*}
    ([u_1:\cdots:u_{n-1}],[v_1 \DT: v_{n-}])\mapsto [u_1v_1:u_1v_2
      \DT: u_{n-1}v_{n-1}],
  \end{equation*}
  %*
  If the coordinate ring of the target is $\bC[d_{ij}, 1\le i,j\le
    n-1]$, the image of $(u_i,v_j)$ is $d_{ij}$. It is well-known
  (e.g.\ see e.g.\ \cite[p.~14]{CST}), that the coordinate ring of the
  image is $S = \bC[d_{ij}]/I_2$, where $I_2$ is the ideal of all $2
  \times 2$-minors, which finishes the proof.
\end{proof}
  
The Hilbert series of $\A_n$ is given in \cite[p.~53]{BC} and is equal
to
%*
\begin{equation}\label{Eq:Hilb}
  \mathcal H = \sum_{d=0}^\infty\binom{d+n-1}{n-1}^2t^d.
\end{equation}
%*
  
It is also known that $\A_n$ is both Gorenstein and Koszul. The
Gorenstein property was first proved in \cite{GW}; the Koszul property
was first settled in \cite{BM}; see also \cite{BC}.

Consider now the map $\Theta_n: \bC[x_{(\si,\tau)}, \si, \tau \in S_n]
\to \bC[\zz,\bzz]$ which sends each variable $x_{(\si,\tau)}$ to
${\nu_2}_{(\si,\tau)}$. This map is graded and doubles the degree.

\begin{proposition}\label{Pp:PhiViaM}
  Let $n \ge 4$, $1 \le i, j \le n-1$ and
  %*
  \begin{equation}\label{Eq:PhiViaM}
    \rho_{ij} \bydef \frac12 (x_{(1i)(1n),(1j)(2,n-1)(2n)} -
    x_{(1i),(1j)(2,n-1)(2n)} + x_{(1i)(1n),(1j)(2n)} -
    x_{(1i),(1j)(2n)})
  \end{equation}
  %*
  Then $\Theta_n(\rho_{ij}) = \BM_{ij} \in M_n$.
\end{proposition}

\begin{proof}
  For $i=j=1$ the proof is an immediate check.
  For any other $i$ and $j$, one has
  %*
  \begin{equation*}
    \rho_{ij} = \mR[((1i),(1j))] \rho_{11},
  \end{equation*}
  %*
  where $\mR$ is a regular representation of $S_n \times S_n$ in
  $\bC[x_{(\si,\tau)}, \si, \tau \in S_n]$ given by
  %*
  \begin{equation*}
    \mR[(\si',\tau')](x_{\si, \tau}) = x_{\si' \si, \tau'\tau}.
  \end{equation*}
  %*
  Recall that by $\mP$ we denote a $n$-dimensional permutation
  representation of $S_n$; then one has
  %*
  \begin{equation*}
    \Theta_n(\rho_{ij}) = \mP[(1i)] \Theta_n(\rho_{11}) \mP[(1j)] =
    \mP[(1i)] \BM_{11} \mP[(1j)] = \BM_{ij}.
  \end{equation*}
  %*
\end{proof}

The kernel $J_n\bydef \name{Ker}(\Theta_n) \subset \bC[S_n\times S_n]$
is an ideal which we call the {\em ideal of relations}. Obviously
$J_n$ is a homogeneous ideal: $J_n = \bigoplus_k J_{nk}$ where $J_{nk}
\bydef \name{Ker} \left.\Theta_n\right|_k$ is the kernel of $\Theta_n$
restricted to the degree $k$ component of the polynomial ring
$\bC[x_{(\si,\tau)}, \si, \tau \in S_n]$.

The condition $x = \sum_{\si, \tau \in S_n} u_{\si,\tau} x_{\si,\tau}
\in J_{n1}$ means that for any $i = 1 \DT, n$, one has
%*
\begin{align*}
  0 &= \Theta(x)(e_i) = \sum_{\si, \tau \in S_n} u_{\si,\tau}
  \mP[\si^{-1}(C - C^{-1}](e_{\tau(i)})\\
  &= \sum_{\si, \tau \in S_n} u_{\si,\tau} \mP[\si^{-1}]
  (e_{\tau(i)+1} - e_{\tau(i)-1}) \\  
  &\text{ (meaning addition and subtraction modulo $n$)}\\
  &= \sum_{\si, \tau \in S_n} u_{\si,\tau} u_{\si,\tau}
  (e_{\si^{-1}(\tau(i)+1)} - e_{\si^{-1}(\tau(i)-1)}).
\end{align*}
%*
In other words, this equality means that for all $i, j = 1 \DT, n$, one
has
%*
\begin{equation}\label{Eq:LinRel}
  \sum_{\si,\tau: \si(j) = \tau(i)+1} u_{\si,\tau} = \sum_{\si,\tau: \si(j) = \tau(i)-1} u_{\si,\tau}.
\end{equation}
%*

Propositions \ref{prop:subalgebra} and
\ref{Pp:PhiViaM} imply the following. 

\begin{corollary}\label{cor:another basis}
  The ideal of relations is generated by all linear elements $x =
  \sum_{\si,\tau \in S_n} u_{\si,\tau} x_{\si,\tau}$, where the 
  coefficients $u_{\si,\tau}$ satisfy equations \eqref{Eq:LinRel} and
  the quadratic elements $\rho_{i_1,j_1}\rho_{i_2,j_2} -
  \rho_{i_1,j_2}\rho_{i_2,j_1}$, where the elements $\rho_{ij}$ are defined
  by equation \eqref{Eq:PhiViaM}.
\end{corollary}

\begin{corollary}[of Corollary \ref{cor:another basis}]
  The Galois group of the Galois closure of the field extension
  $\bC[\zz,\bzz]^{S_n \times S_n}(\nu_2): \bC[\zz,\bzz]^{S_n \times
    S_n}$ consists of all maps $\gamma: S_n \times S_n \to S_n \times
  S_n$ such that the linear transformation sending $x_{\si,\tau}
  \mapsto x_{\gamma(\si,\tau)}$ for all $\si, \tau \in S_n$ preserves
  all the relations described in Corollary \ref{cor:another basis}.
\end{corollary}

\section{Examples and illustrations: triangle}\label{sec:triangles}

In this section we illustrate our general results in the
simplest nontrivial case $n=3$, i.e.\ when the considered polygons are triangles.

\subsection{$S_3$-action and the Galois group}\label{Sec:Triangle}

First of all, for $n=3$, the numerator of equation \eqref{eq:PsiN} is
a constant, so it is equal to $\nu_2 = \nu_2(\zz,\bzz)$. Therefore,
 harmonic moments of a triangle are related as
%*
\begin{equation}\label{Eq:RelM3}
  \frac{\nu_{j+2}}{\nu_2} = \binom {j+2}{2} h_j(z_1,z_2,z_3),
\end{equation}
%*
where $h_j(z_1,z_2,z_3)$ denotes the complete symmetric function of
degree $j$ in three variables, that is, the sum of all monomials of
degree $j$ in $z_1, z_2, z_3$. So $\nu_2$ will be playing a crucial
role in the following considerations. Denote $M \bydef \nu_2$ for
short (the same thing was $M_{\id,\id}$ in Section \ref{Sec:Both}: we
again do not distinguish bilinear forms from their matrices).

\begin{theorem}\label{th:triangle}
  The generators $M = \nu_2, \nu_3, \bar \nu_3, \nu_4, \bar \nu_4, \nu_5,
  \bar \nu_5$ of the field $\widetilde \F_3$ satisfy a sole
  relation $L(M, e_1 \DT, \bar e_3) = 0$,  where $L \bydef
  \name{Res}_S(R,Q)$.   Here $R = 16M^2 + \det
  \Omega(S)$ with
  %*
  \begin{equation*}
    \Omega(S) = \begin{pmatrix}
      3 & e_1 & \bar e_1\\
      e_1 & e_1^2-2e_2 & S\\
      \bar e_1 & S &\bar e_1^2- 2\bar e_2
    \end{pmatrix}
  \end{equation*}
  %*
  and 
  %*
  \begin{equation*}
    Q = \prod_{\si \in S_3} \left(S-z_1\bar z_{\si(1)}-z_2\bar z_{\si(2)}-z_3\bar z_{\si(3)}\right).
  \end{equation*}
  %*
\end{theorem}
(Here $\name{Res}_S(R,Q)$ denotes the resultant of polynomials $R$ and
$Q$ with respect to the variable $S$).

\begin{remark}
  Explicitly, one has $R = -3S^2+2e_1\bar e_1 S+16M^2+e_1^2\bar
  e_1^2-4e_1^2\bar e_2-4\bar e_1^2 e_2+12e_2\bar e_2$. $Q$ is a
  polynomial of degree $6$ with respect to $S$; it is symmetric in the
  $z_i$ and the $\bar z_i$ separately. Hence $Q$ can be regarded as a
  polynomial of degree $6$ in $S$ with the coefficients being polynomials
  in  the variables $e_k$ and $\bar e_k$, $k = 1, 2, 3$. The total degree
  of $Q$ is $20$; it contains $66$ terms.
\end{remark}

\begin{proof}[Proof of Theorem~\ref{th:triangle}]
By Theorem~\ref{Th:Formula},  $M = \frac{1}{2i} \det
  \omega$, where
  %*
  \begin{equation}\label{Eq:DetNu2}
    \omega = \begin{pmatrix}
      1 & 1 & 1\\
      z_1 & z_2 & z_3\\
      \bar z_1 & \bar z_2 & \bar z_3
    \end{pmatrix}.
  \end{equation}
  %*

  We will follow the argument suggested by R.\,Bryant in
  \cite{br}. One has $\omega \cdot \omega^* = \Omega(S)$, where $S =
  z_1\bar z_1+z_2\bar z_2+z_3\bar z_3$. Thus $16M^2 = -\det \Omega(S)$
  for this value of $S$.  The same value of $S$ is a root of the
  polynomial $Q$, so $L = \name{Res}_S(R,Q) = 0$. An explicit formula
  for the resultant shows that $\name{Res}_S(R,Q)$ has degree $12$
  with respect to $M$. Theorem~\ref{Th:Both} implies that $L$ is the
  minimal polynomial for $M$.
  \end{proof}

\begin{remark}
  Combining Theorem \ref{th:triangle} with equations \eqref{Eq:RelM3}, 
  one obtains a relation among $\splitatcommas{\nu_2, \nu_3, \nu_4, \nu_5,
    \bar \nu_3, \bar \nu_4, \bar \nu_5}$. We calculated it explicitly using
  {\em Macaulay} computer algebra system. The result is a very long
  polynomial with integer coefficients (of the order of several
  millions) which is weighted homogeneous of degree $64$ with $\nu_k$
  and $\bar \nu_k$ having weight $k$ for $k = 2, 3, 4, 5$.
\end{remark}
  
Let us now present the relations between $M_{(\si,\tau)}$. By equation
\eqref{Eq:MinimalPol} one should take one pair $(\si,\tau)$ for every
right coset of $S_3 \times S_3$ with respect to the cyclic group
generated by $(C,C)$, where $C$ is the cyclic shift $(123)$. The
number of these cosets is $3!2!  = 12$, and a convenient system of
representatives is $\{(\si,\tau) \mid \si \in \{\id, (12)\}, \tau \in
S_3\}$.

The vector space spanned by $M_{(\si,\tau)}$ has dimension $(3-1)^2 =
4$. So, there exist $12-4=8$ independent linear relations between
$M_{(\si,\tau)}$. Of them, $6$ are two-term:
%*
\begin{equation}\label{Eq:2term}
  M_{(12),(12)\tau} + M_{\id,\tau} = 0, \quad \tau \in S_3, 
\end{equation}
%*
and the additional two  are three-term:
%*
\begin{equation}\label{Eq:3term}
  \begin{aligned}
    &M_{\id,\id} + M_{\id,(123)} + M_{\id,(132)} = 0, \\
    &M_{(12),\id} + M_{(12),(123)} + M_{(12),(132)}
    = 0.
  \end{aligned}
\end{equation}
%*

The basis in the image of the map $\Theta_3$ is formed by $4$ vectors,
$\BM_{11}$, $\BM_{12}$, $\BM_{21}$, and $\BM_{22}$. For $n=3$, all
quadratic relations \eqref{Eq:Rel} reduce to only one:
%*
\begin{equation}\label{Eq:Quadr3}
  \BM_{12} \BM_{21} = \BM_{11} \BM_{22}.
\end{equation}
%*

Direct computation shows that for $n=3$ the forms $\BM_{ij}$ can be
expressed via $M_{\si,\tau}$ as follows (recall that the general
formulas \eqref{Eq:PhiViaM} work only for $n \ge 4$):
%*
\begin{align*}
  \BM_{11} &= \frac13 (M_{\id,\id} + 2M_{\id,(123)} - M_{\id,(12)} -
  M_{\id,(23)}),\\
  \BM_{12} &= \frac13 (2M_{\id,\id} + M_{\id,(123)} - M_{\id,(12)} -
  2M_{\id,(23)}),\\
  \BM_{21} &= \frac13 (-M_{\id,\id} + M_{\id,(123)} - M_{\id,(12)} -
  2M_{\id,(23)}),\\
  \BM_{22} &= \frac13 (M_{\id,\id} + 2M_{\id,(123)} - 2M_{\id,(12)} -
  M_{\id,(23)}).
\end{align*}
%*
Substitution of these formulas into the quadratic relation
\eqref{Eq:Quadr3} gives
%*
\begin{equation}\label{Eq:QuadrRel}
  M_{\id,\id}^2 + M_{\id,\id}M_{\id,(123)} + M_{\id,(123)}^2 =
  M_{\id,(12)}^2 + M_{\id,(12)}M_{\id,(23)} + M_{\id,(23)}^2.
\end{equation}
%*

The Galois group $G_3$ of the equation \eqref{Eq:MinimalPol} permutes
its $12$ roots $\splitatcommas{M_{\id,\tau}, M_{(12),\tau}, \tau \in
  S_3}$ preserving the linear relations \eqref{Eq:2term} and
\eqref{Eq:3term} together with the quadratic relation
\eqref{Eq:QuadrRel}. Thus $G_3 \subset S_{12}$.

For $\ga \in G_3$, relations \eqref{Eq:2term} imply that
there exists a bijection $\tilde\ga: S_3 \to S_3$ and a map $\eps: S_3
\to \{1, -1\}$ such that $\ga(M_{\id,\tau}) = \eps[\tau]
M_{\id,\tilde\ga[\tau]}$ for all $\tau \in S_3$. Then it follows  from
\eqref{Eq:2term} that $\ga(M_{(12),\tau}) = -\eps[\tau]
M_{\id,(12)\tilde\ga[\tau]} = \eps[\tau] M_{(12),\tilde\ga[\tau]}$ which means that 
the bijection $\tilde\ga \in S_6$ and the map $\eps$ determine $\ga$
uniquely. In other words, $G_3$ is a subgroup of the Coxeter group
$B_6$ of signed permutations (which is naturally embedded into
$S_{12}$, as described above).

Further, to preserve relations \eqref{Eq:3term} the map $\tilde\ga$
should either map the subsets $A_3 \bydef \{\id,(123),(132)\}$ and
$S_3 \setminus A_3 = \{(12),(13),(23)\}$ of $S_3$ to themselves or to
each other. In both cases, the numbers $\eps[\tau]$ should remain the
same while $\tau$ is changing within a set. The pairs $(\tilde\ga,
\eps) \in G_3$ where $\tilde\ga$ preserves the sets form a subgroup
$G_3^+ \subset G_3$ of index $2$.

Notice now that the quadratic form $Q(u) = u_1^2 + u_1 u_2 + u_2^2$ is
$S_3$-invariant on the subspace $V_3 \bydef \{u_1 e_1 + u_2 e_2 + u_3 e_3
\mid u_1 + u_2 + u_3 = 0\} \subset \bC^3 = \langle e_1, e_2,
e_3\rangle$ with the permutation action of the $S_3$. (This can be
checked by an easy computation; actually, up to a
factor, this form is equal to the restriction of the form $u_1^2 + u_2^2 +
u_3^2$ defined in $\bC^3$ to $V_3$). So, any mapping $\ga$ described above
automatically preserves relation \eqref{Eq:QuadrRel}. Hence, the
subgroup $G_3^+$ consists of {\em all} pairs $(\tilde\ga, \eps)$ where
$\tilde\ga$ preserves the sets $A_3$ and $S_3 \setminus A_3$ and
$\eps$ is constant within either set; thus, $G_3^+$ is isomorphic to
the group $S_3 \times S_3 \times \bZ_2 \times \bZ_2$ and contains
$144$ elements. The whole group $G_3$ contains $288$ elements and is a
semi-direct product of $G_3^+$ and the $2$-element group $\bZ_2$.

\subsection{Graphic presentation of the moment $M=\nu_2$}

It follows from \eqref{Eq:RelM3} that to analyze the moments for $n=3$
it is enough to study the lowest moment $\nu_2$ defined by the equation
\eqref{Eq:DetNu2}.

To represent the points $(x_1,y_1), (x_2,y_2), (x_3,y_3) \in \bC^2$,
let us draw two triples of complex numbers: $\zz = (z_1, z_2, z_3)$
and $\bzz = (\bar z_1, \bar z_2, \bar z_3)$ where $z_j = x_j + iy_j$
and $\bar z_j = x_j - iy_j$, $j = 1, 2, 3$. For generic choice of
$z_j$, $\bar z_j$ there exist unique numbers $\alpha, \beta \in \bC$
such that $\alpha z_1 + \beta = \bar z_1$ and $\alpha z_2 + \beta =
\bar z_2$ and therefore
%*
\begin{equation}\label{Eq:AreaExpl}
  \nu_2 = \frac{1}{2i} \det \begin{pmatrix}
    1 & 1 & 1 \\
    z_1 & z_2 & z_3\\
    0 & 0 & \bar z_3 - (\alpha z_3 + \beta)
  \end{pmatrix} = \frac{1}{2i}(\bar z_3 - (\alpha z_3 + \beta))(z_2-z_1).
\end{equation}
%*
So if $w = \alpha z_3 + \beta$ then the triangle $\bar z_1 \bar z_2 w$
is similar to the triangle $z_1z_2z_3$ where the similarity map sends
$z_1 \mapsto \bar z_1, z_2 \mapsto \bar z_2, z_2 \mapsto
w$. Obviously, this condition determines $w$ uniquely. Then the vector
connecting $\bar z_3$ and $w$ represents the complex number $\bar z_3
- (\alpha z_3 + \beta)$. Thus it follows from \eqref{Eq:AreaExpl} that
the moment $\nu_2$ is the product of this number by the complex number
represented by the vector joining the vertices $z_1$ and $z_2$,
divided by $2i$. In particular, $\lmod\nu_2\rmod$ is one half of the
product of the lengths of these two vectors. Thus, $\nu_2$ can be
thought as a measure of non-similarity of two triangles.

The action of the group $S_3 \times S_3$ preserves triples $\zz$ and
$\bzz$, but changes the numbering of these points. Identity
\eqref{Eq:3term} now involves moments $\nu_2$ calculated using
\eqref{Eq:AreaExpl} with the same $z_1, z_2, z_3$ in all three terms
and $\bar z_1, \bar z_2, \bar z_3$ changing their labels in a
cycle.

If the vertices of the triangle are real then \eqref{Eq:3term}
translates into a statement from the elementary Euclidean
geometry. Namely, denote by $A_j$ the point $\bar z_j \in \bC = \bR^2$
and by $C_j$, the point $\alpha z_j + \beta$ from \eqref{Eq:AreaExpl};
here $j = 1,2,3$. Then \eqref{Eq:3term} and \eqref{Eq:AreaExpl} give:

\begin{Theorem}
  Let $A_1A_2A_3$ be a triangle in the plane $\bR^2$. Let $C_1, C_2,
  C_3 \in \bR^2$ be points such that the triangles $A_1C_3A_2$,
  $A_2C_1A_3$ and $A_3C_2A_1$ are similar with the similarity maps
  sending vertices to vertices as written (e.g.\ $A_1 \mapsto A_2, A_2
  \mapsto A_3, C_3 \mapsto C_1$ for the first two triangles,
  etc.). Then the sum of the vectors
  $\widevec{A_1C_1}+\widevec{A_2C_2}+\widevec{A_3C_3}$ vanishes.
\end{Theorem}

\vspace*{4cm}
\setlength{\unitlength}{0.25cm}
\thicklines
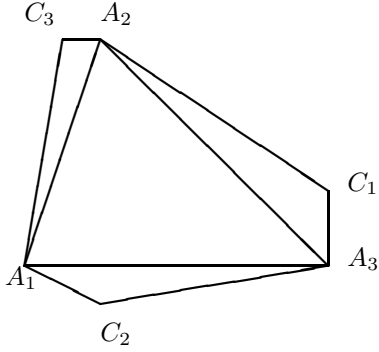
\begin{picture}(0,0)
  \put(0,0){\line(1,0){16}}
  \put(0,0){\line(1,3){4}}
  \put(4,12){\line(1,-1){12}}
  \put(0,0){\line(1,6){2}}
  \put(2,12){\line(1,0){2}}
  \put(4,12){\line(3,-2){12}}
  \put(16,4){\line(0,-1){4}}
  \put(0,0){\line(2,-1){4}}
  \put(4,-2){\line(6,1){12}}
  \put(-1,-1){$A_1$}
  \put(4,13){$A_2$}
  \put(17,0){$A_3$}
  \put(0,13){$C_3$}
  \put(17,4){$C_1$}
  \put(4,-4){$C_2$}
  \put(20,12){$A_1C_3A_2 \sim A_2C_1A_3 \sim A_3C_2A_1$}
  \put(20,9){$\widevec{A_1C_1} + \widevec{A_2C_2} + \widevec{A_3C_3} =0$}
  \put(20,0){Identity \eqref{Eq:3term} from the geometric point of view.}
\end{picture}
\vspace*{2cm}

\section{Further outlook}\label{sec:outlook}

\def \thesubsubsection {\arabic{subsubsection}}

\subsubsection{} According to assertion \ref{It:GenNu2} of
Theorem~\ref{Th:Both} each moment $\nu_j(\zz, \bzz)$ is a rational
function of $e_1(\zz) \DT, e_n(\zz), e_1(\bzz) \DT, e_n(\bzz)$, and
$\nu_2(\zz,\bzz)$. Is it possible to find these rational functions
explicitly?

\subsubsection{} The main motivation for the present paper comes from
a recent article \cite{KoShSt} by the third author (joint with C.~Kohn
and B.~Sturmfels) where general (not necessarily harmonic) moments for
convex polytopes were considered. In particular, \cite{KoShSt}
contains a complete description of relations between the axial moments
of such polytopes. A similar problem for fields and
rings of general moments is still widely open and is
apparently closely related to complicated questions about the ring of diagonal
harmonics defined in \cite{Haiman}.


\begin{thebibliography}{99}

\bibitem{BM} S.~Barcanescu, N.~Manolache, Nombres de Betti d'une
  singularit\'e de Segre--Veronese. (French) C.\ R.\ Acad.\ Sci.\ Paris
  S\'er.\ A-B 288 (1979), no.\ 4, A237--A239.

\bibitem{BC} W.~Bruns, A.~Conca, Gr\"obner bases and determinantal
  ideals. Commutative algebra, singularities and computer algebra
  (Sinaia, 2002), 9--66, NATO Sci.\ Ser.\ II Math.\ Phys.\ Chem., 115,
  Kluwer Acad.\ Publ., Dordrecht, 2003.

\bibitem{br} R.~Bryant,
  https://mathoverflow.net/questions/81690/area-of-triangle-from-coefficients-of-its-cubic

\bibitem{broSt} M.~A.~Brodsky, and V.~N.~Strakhov, On the uniqueness
  of the inverse logarithmic potential problem, SIAM Journal on
  Applied Mathematics, vol.~46(2) (1986), 324--344.
  
\bibitem{CST} A.~Conca, S.~Ho\c{s}ten, R.~Thomas, Nice initial
  complexes of some classical ideals. Algebraic and geometric
  combinatorics, 11--42, Contemp.\ Math., 423, Amer.\ Math.\ Soc.,
  Providence, RI, 2006.

\bibitem{Da} Ph.~J.~Davis. Triangle formulas in the complex
  plane. Math.\ Comp., vol.~18 (1964), 569--577.

\bibitem{GMV} G.~Golib, P.~Milanfar, and J.~Varah, A stable numerical
  method for inverting shape from moments, SIAM
  J.\ SCI.\ Comput.\ vol.~21(4) (1999), 1222--1243.
  
\bibitem{GW}  Sh.~Goto,  K.~Watanabe, On graded rings. I.\ J.\ Math.\ Soc.\ Japan 30 (1978), no.\ 2, 179--213.

\bibitem{GPSS} N.~Gravin, D.~V.~Pasechnik, B.~Shapiro, and M.~Shapiro,
  On moments of a polytope, arXiv:1210.3193, Analysis and Math.\ Phys.,
  DOI: 10.1007/s13324-018-0226-8.

\bibitem{Haiman} M.~Haiman, Conjectures on the quotient ring by
  diagonal invariants, J.\ Algebraic Combin.\ 3 (1994), 17--76.

\bibitem{KoShSt} K.~Kohn, B.~Shapiro, and B.~Sturmfels, Moment
  varieties of measures on polytopes, Annali della Scuola Normale
  Superiore di Pisa, to appear.

\bibitem{LaPu} J.-B.~Lasserre, M.~Putinar, Algebraic-exponential Data
  Recovery from Moments, Discrete Comput Geom (2015) vol.~54, 993--1012.
    
  \bibitem{MWZ} A.~Marshakov, P.~Wiegmann, A.~Zabrodin, Integrable
    structure of the Dirichlet boundary problem in two dimensions,
    Comm.\ Math.\ Phys.\ 227 (2002), no.~1, 131--153. 

\bibitem{Mumford} D.~Mumford, Algebraic Geometry I: Complex Projective
  Varieties, Springer Science \& Business Media, 1995, 186 pp.

\bibitem{2DToda} S.~Natanzon, A.~Zabrodin, Symmetric Solutions to
  Dispersionless 2D Toda Hierarchy, Hurwitz Numbers, and Conformal
  Dynamics, International Mathematics Research Notices, Volume 2015,
  Issue 8, 2015,  2082--2110. 

\bibitem{Nov} P.~S.~Novikov, On the uniqueness of the solution of the
  inverse potential problem, Doklady AN SSSR, vol.~18 (1938),
  165--168. (In Russian.)

\bibitem{PS} D.~Pasechnik, and B.~Shapiro, On polygonal measures with
  vanishing harmonic moments, Journal d'Analyse math\'ematique,
  vol.~123(1) (2014), 281--301.

\bibitem{St} R.~Stanley, Enumerative combinatorics. Volume 1. Second
  edition. Cambridge Studies in Advanced Mathematics, 49. Cambridge
  University Press, Cambridge, 2012. xiv+626 pp.

\bibitem{Wa} E. Wachspress, A Rational Finite Element basis, Academic
  Press, 1975, 330 pp.

\bibitem{War} J. Warren, barycentric coordinates for convex polytopes,
  Advances in Computational Mathematics, vol.~6(1) (1996), 97--108.

\bibitem{Repr} J.H.M. Wedderburn, On Hypercomplex Numbers,
  Proc.\ London Math.\ Soc., vol.~2 (1908), no.~6,. 77--118.
   
\end{thebibliography}
\end{document}